\documentclass[12pt]{article}
\usepackage{amsmath,amssymb,amsthm}
\usepackage{hyperref}
\newcommand{\tmmathbf}[1]{\ensuremath{\boldsymbol{#1}}}

\newtheorem{theorem}{Theorem}[section]
\newtheorem{corollary}[theorem]{Corollary}
\newtheorem{lemma}[theorem]{Lemma}

\numberwithin{equation}{section}

\usepackage{amssymb,latexsym,amsmath,epsfig,amssymb,amsthm, amsmath,color}

\def\F{\mathcal{F}}

\def\P{\mathcal{P}}
\def\L{\mathcal{L}}
\def\C{\mathcal{C}}

\def\R{\mathcal{R}}
\def\E{\mathcal{E}}

\def\A{\mathcal{A}}
\def\C{\mathcal{C}}

\begin{document}

\title{Some combinatorial number theory problems over finite valuation rings}
\author{
    Pham Van Thang\thanks{EPFL, Lausanne, Switzerland. This research was partially supported by Swiss National Science Foundation grants 200020-162884 and 200020-144531
    Email: {\tt thang.pham@epfl.ch}}
  \and
    Le Anh Vinh\thanks{University of Education, Vietnam National University Hanoi. Email: {\tt vinhla@vnu.edu.vn
}}}
\date{}
\maketitle
\begin{abstract}
Let $\R$ be a finite valuation ring of order $q^r$. In this paper we generalize and improve several well-known results, which were studied over finite fields $\mathbb{F}_q$ and finite cyclic rings $\mathbb{Z}/p^r\mathbb{Z}$, in the setting of finite valuation rings.
\end{abstract}
\section{Introduction}
\subsection{Dot-product congruence classes of simplices}
Let $\mathbb{F}_q$ be a finite field of order $q$ with $q=p^n$ for some prime $p$ and positive integer $r$.  We say that two $k$-simplices  in $\mathbb{F}_q^d$ with vertices $(\mathbf{x}_1, \ldots, \mathbf{x}_{k+1})$, $(\mathbf{y}_1, \ldots, \mathbf{y}_{k+1})$ are in a congruence class if the following condition satisfies
\begin{equation}\label{helloworld}||\mathbf{x}_i-\mathbf{x}_j||=||\mathbf{y}_i-\mathbf{y}_j||, \quad ~ 1\le i, j\le k+1.\end{equation}

Hart and Iosevich \cite{hart1} made the first investigation on counting the number of congruence classes of simplices determined by a point set in $\mathbb{F}_q^d$. More precisely, they proved that if $|\E|\gg q^{\frac{kd}{k+1}+\frac{k}{2}}$ with $d\ge \binom{k+1}{2}$, then $\E$ contains a copy of all $k$-simplices with non-zero edges. Several progress on improving this exponent have been made in recent years, for instance, Chapman et al. \cite{chapman} indicated that one can get a positive proportion of all $k$-simplices in $\mathbb{F}_q^d$ under the condition $|\E|\gg q^{\frac{d+k}{2}}$, and Bennett et al. \cite{groupaction} improved this condition to $q^{d-\frac{d-1}{k+1}}$.  Here and throughout,  $X \ll Y$ means that there exists $C>0$ such that $X\le  CY$, and $X=o(Y)$ means that $X/Y\to 0$ as $q\to \infty$, where $X, Y$ are viewed as functions in $q$.

A variant of this problem was studied by the second listed author \cite{vinh-spectral} with the condition (\ref{helloworld}) is replaced by  
\begin{equation}\label{helloworld2}\mathbf{x}_i\cdot \mathbf{x}_j=\mathbf{y}_i\cdot\mathbf{y}_j, \quad ~ 1\le i, j\le k+1.\end{equation}

In this case, we say that two $k$-simplices $(\mathbf{x}_1, \ldots, \mathbf{x}_{k+1})$ and $(\mathbf{y}_1, \ldots, \mathbf{y}_{k+1})$ are in a dot-product congruence class.

The author of \cite{vinh-spectral} proved that if $|\E|\gg q^{\frac{d+k}{2}}$, then the number of dot-product congruence classes of $k$-simplices in $\E$ is at least $(1-o(1))q^{\binom{k+1}{2}}$. This is also an extension of \cite[Theorem 1.4]{hart1}, and is the best known result sofar.  We remark here that the condition (\ref{helloworld}) is equivalent to the fact that there exist $\theta\in O(d, \mathbb{F}_q)$ (orthogonal group in $\mathbb{F}_q^d$) and $\mathbf{z}\in \mathbb{F}_q^d$ so that $\mathbf{z}+\theta(\mathbf{x}_i)=\mathbf{y}_i$ for $i=1, 2, \ldots, k+1$. From this fact, the authors of \cite{groupaction} used ingenious arguments by combining elementary results from group action theory and Fourier analytic methods to get the threshold $q^{d-\frac{d-1}{k+1}}$. However, this approach does not work for the case of dot-product congruence classes of simplices, since we can not guarantee that there exist $\theta\in O(d, \mathbb{F}_q)$ and $\mathbf{z}\in \mathbb{F}_q^d$ so that $\mathbf{z}+\theta(\mathbf{x}_i)=\mathbf{y}_i$ for $i=1, 2, \ldots, k+1$ when two simplices are in a  dot-product congruence class. 

For the case $k=1$ and $d=2$, it has been shown that if $|\E|\gg q^{4/3}$, then the number of congruence classes of $1$-simplices in $\E$ (distinct distances) is at least $\gg q$. However, for the dot-product case, the best known exponent on the cardinality of $\E$ to get $\gg q$ dot-product congruence classes of $1$-simplices in $\E$ (distinct dot product values) is $q^{3/2}$. If we assume that any line passing through the origin contains at most $|\E|^{1/2}$ points, then the exponent $q^{4/3}$ also holds for the dot-product problem, see \cite[Theorem 2.2, Chapter 2]{iosevich-book} for more details. For general cases, improving the threshold $q^{3/2}$ to $q^{4/3}$ is still an open question.

Let $\R$ be a finite valuation ring of order $q^r$, where $q=p^n$ is an odd prime power. Throughout, $\R$ is assumed to be commutative, and to have an identity. Let us denote the set of units, non-units in $\R$ by $\R^{*}, \R^{0}$ respectively. The detailed definition of finite valuation rings can be found in \cite{bogan, htv}. Note that finite fields and finite cyclic rings are special cases of finite valuation rings.

The initial result on the dot product problem in setting of finite valuation rings was given by Nica in \cite{bogan}. The precise statement is as follows.
\begin{theorem}[\textbf{Nica}, \cite{bogan}]\label{nc}
Let $\E, \F$ be two sets in $\R^d$. For any $\lambda\in \R^*$, let $N_{\lambda}(\E, \F)$ be the number of pairs $(\mathbf{a}, \mathbf{b})\in \E\times \F$ satisfying $\mathbf{a}\cdot \mathbf{b}=\lambda$. Then we have the following estimate
\[\left\vert N_{\lambda}(\E, \F)-\frac{|\E||\F|}{q^r}\right\vert\le q^{(d-1)(r-1/2)}\sqrt{|\E||\F|}.\]
\end{theorem}
Theorem \ref{nc} implies that if $|\E||\F|\ge q^{d(2r-1)+1}$, then for any $\lambda\in \R^*$, the equation $\mathbf{a}\cdot\mathbf{b}=\lambda$ is solvable with $a\in \E, b\in \F$.

Motivated by this result, in this paper we prove the following result on the number of dot-product congruence classes of simplices over finite valuation rings.
\begin{theorem}\label{muoi}
Let $\R$ be a finite valuation ring of order $q^r$. Given a set $\mathcal{E}\subseteq \R^d$.  Suppose that 
\[|\mathcal{E}|\gg q^{\frac{(d-1)(2r-1)+r(k+1)}{2}}.\]
Then the number of dot-product congruence classes of $k$-simplices  in $\E$ is at least $(1-o(1))q^{r\binom{k+1}{2}}$.
\end{theorem}
\subsection{An improvement on the number of triangle areas}
For $\E\subseteq \R^d$, we define
\begin{equation}\label{eqdt1}V_d(\E):=\left\lbrace \det\left(\mathbf{x}^1-\mathbf{x}^{d+1},\ldots, \mathbf{x}^d-\mathbf{x}^{d+1}\right)\colon \mathbf{x}^i\in \E,~1\le i\le d+1\right\rbrace\end{equation}
as the set of $d$-dimensional volumes determined by $\P$, and the set of pinned volumes at  a point $\mathbf{z}\in \E$
\begin{equation}\label{eqdt2}V_d^{\mathbf{z}}(\E):=\left\lbrace \det\left(\mathbf{x}^1-\mathbf{z},\ldots, \mathbf{x}^d-\mathbf{z}\right)\colon \mathbf{x}^i\in \E,~1\le i\le d\right\rbrace.\end{equation}

In \cite{area}, Iosevich,  Rudnev, and Zhai showed that if $|\E|\ge 64q\log q$, then there exists a point $\mathbf{z}\in \E$ such that $|V_2^\mathbf{z}(\E)|\ge q/2$. The finite cyclic ring analogue of this problem  is recently investigated by Yazici \cite{ya}. In particular, she proved that for $\E\subseteq \mathbb{Z}/p^r\mathbb{Z}$, if $|\E|\ge p^{2r-(1/2)}$ then  $|V_2(\E)|\ge \frac{p^r}{4}\frac{1+p}{p}-1$.  This implies that if $r=1$, Yazici's bound is  weaker than that of \cite{area}. In this section, we will give an improvement of these results in the setting of finite valuation rings.

\begin{theorem}\label{mainthm297}
Let $\R$ be a finite valuation ring of order $q^r$. Let $\mathcal{E}$ be a set of points in $\R^2$. If $q^{2r-1}=o(|\E|)$, then there exists $\mathbf{z}\in \E$ such that $|V_2^{\mathbf{z}}(\E)|= (1-o(1))q^r$.
\end{theorem}

Theorem \ref{mainthm297} implies that when $\R$ is a finite field, i.e. $r=1$, in order to get $(1-o(1))q$ distinct areas we only the condition $q=o(|\E|)$. This improves the threshold $64q\log q$ on the cardinality of $\E$ in \cite{area}. When $\R$ is a finite cyclic ring, i.e. $q$ is a prime, the bound $p^{2r-\frac{1}{2}}$ in \cite{ya} is decreased to $p^{2r-1}$.

By using inductive arguments, one can obtain a similar result for higher dimensional cases, which is also a generalization of the main result in \cite{vinh-areas}.
\begin{theorem}\label{co2}
Let $\R$ be a finite valuation ring of order $q^r$, and let $\E$ be a set of points in $\R^d$. If $q^{r-1}\cdot q^{r(d-1)}=o(|\E|)$, then 
then there exists a point $\mathbf{z}\in \E$ such that $|V_d^{\mathbf{z}}(\E)|= (1-o(1))q^r$.
\end{theorem}

\subsection{An improvement on permanents of matrices}
Let $M$ be an $k\times k$ matrix. The permanent of $M$ is defined by 
\[\mathtt{Per}(M):=\sum_{\sigma\in S_k}\prod_{i=1}^ka_{i\sigma(i)}.\]
For $\E\subseteq \mathbb{F}_q^k$, let $M_k(\E)$ denote the set of $k\times k$ matrices with rows in $\E$, and $P_k(\E)=\{\mathtt{Per}(M)\colon M\in M_k(\E)\}$. The author of \cite{vinh-perma} proved that for any $\A\subseteq \mathbb{Z}/p^r\mathbb{Z}$, if
\[|\A|\gg rp^{r-\frac{1}{2}+\frac{1}{2k}},\]
then $\left(\mathbb{Z}/p^r\mathbb{Z}\right)^*\subseteq P_k(\A^k)$ with $\gcd(k, p^r)=1$, where $\left(\mathbb{Z}/p^r\mathbb{Z}\right)^*$ is the set of all units in $\mathbb{Z}/p^r\mathbb{Z}$

In this paper, we are able to improve the threshold $q^{r-\frac{1}{2}+\frac{1}{2k}}$ to $q^{\frac{(k-1)(2r-1)+r}{2k-1}}$ as in the following theorem. 
\begin{theorem}\label{perper}
Let $\R$ be a finite valuation rings of order $q^r$, and $k$ be an integer with $\gcd(k, q^r)=1$. For any $\A\subset \R$, if 
\[|\A|\gg q^{\frac{(k-1)(2r-1)+r}{2k-1}},\]
then $|P_k(\A^k)|=(1-o(1))q^r$.
\end{theorem}
\subsection{Monochromatic sum-product}
For $A_1, A_2\subset \mathbb{F}_p$, where $p$ is a prime, Shkredov \cite{shkredov} showed that $|A_1| |A_2| \geq 20p$, then there exist $x, y \in \mathbb{F}_p$ such that $x + y \in A_1, x\cdot y \in A_2$. Cilleruelo \cite{cils} extended this result to arbitrary finite fields $\mathbb{F}_q$ of $q$ elements using Sidon sets as follows.
\begin{theorem}[\cite{cils}]\label{cungmau1}
For any $X_1, X_2\subseteq \mathbb{F}_q$ of cardinality $|X_1||X_2|>2q$, there exist $x, y\in \mathbb{F}_q$ such that $x+y\in X_1, x\cdot y\in X_2$.
\end{theorem}
In this section, we extend Theorem \ref{cungmau1} to the setting of finite valuation rings. 
\begin{theorem}\label{r1}
Let $\R$ be a finite valuation ring of order $q^r$.  For any $X_1,X_2 \subseteq \R^{*}$ of cardinality
\[ |X_1||X_2| >\frac{q^{4r-1}}{(q^r-q^{r-1})^2} ,\] there exist $x,y \in \R^{*}$ such that $x+y \in X_1$ and $x\cdot y \in X_2$.
\end{theorem}
The rest of this paper is organized as follows: in Section $2$, we mention some tools from spectral graph theory. The proofs of Theorems \ref{muoi}, \ref{mainthm297}, \ref{perper}, and \ref{r1} are given in Sections $3$-$6$.

\section{Tools from spectral graph theory}
We say that a bipartite graph $G=(A\cup B, E)$ is \emph{biregular} if in both of its two parts, all vertices have the same degree. If $A$ is one of the two parts of a bipartite graph, we write $\deg(A)$ for the common degree of the vertices in $A$.
Label the eigenvalues so that $|\lambda_1|\geq |\lambda_2|\geq \cdots \geq |\lambda_n|$.
Note that in a bipartite graph, we have $\lambda_2 = -\lambda_1$. In this paper, we denote the adjacency matrix of $G$ by $M$. The following is the expander mixing lemma for bipartite graphs. The reader can find a detailed proof in \cite{eustis}.
\begin{lemma}\label{expander} Suppose $G$ is a bipartite graph with parts $A,B$ such that the vertices in $A$ all have degree $a$ and the vertices in $B$ all have degree $b$. For any two sets $X\subset A$, and $Y\subset B$, the number of edges between $X$ and $Y$, $e(X,Y)$, satisfies
\[\left\vert e(X,Y)-\frac{a}{|B|}|X||Y|\right\vert\le \lambda_3\sqrt{|X||Y|},\] where $\lambda_3$ is the third eigenvalue of $G$.
\end{lemma}
The following theorem is an analogue of \cite[Theorem 9.2.4]{as}.
\begin{theorem}\label{expander2}
Let $G=(A\cup B,E)$ be a bipartite graph as in Lemma \ref{expander}, and $U, V$ two subsets in $A, B$, respectively. Then we have the following estimate
\[\sum_{u\in U}\left(N_V(u)-\frac{a}{|B|}|V|\right)^2\le \lambda_3^2|V|,\] where $N_V(u)=N(u)\cap V$, and $N(u)$ is the set of all neighbors of $u$.
\end{theorem}
\begin{proof}
Let denote $c=|V|/|B|$, and $x$ be a vector, where $x_i=I_{i\in V}-c\textbf{1}_B$. We note that $\sqrt{a}\textbf{1}_A\pm \sqrt{b}\textbf{1}_B$ are eigenvectors corresponding to $\lambda_1, \lambda_n$. It follows from the definition of $x$ that $<x,\textbf{1}_A>=0$ and $<x,\textbf{1}_B>=0$. Then $x\in W$, and $||Mx||^2\le \lambda_3^2||x||^2$. We note that \[<Mx,Mx>=\sum_{u\in A}\left(N_V(u)-\frac{a|V|}{|B|}\right)^2,\]
and $||x||^2=(1-c)^2|V|+(|B|-|V|)c^2=(1-c)|V|<|V|$, then the lemma follows from the fact that \[\sum_{u\in U}\left(N_V(u)-\frac{a|V|}{|B|}\right)^2\le \sum_{u\in A}\left(N_V(u)-\frac{a|V|}{|B|}\right)^2.\]
\end{proof}

\subsection{ Product graphs over finite valuation rings}
We define the product graphs $\P_{q,r}(\R)=(A\cup B,E)$ over finite valuation rings $\R$ as follows: $A=B=\R^d\setminus (\R^{0})^d$, and there is an edge between $\mathbf{x}\in A$ and $\mathbf{y}\in B$ if and only if $\mathbf{x}\cdot\mathbf{y}=1$. The spectrum of this graph was given by Nica \cite{bogan}.
\begin{theorem}[\textbf{Nica}, \cite{bogan}]\label{bogan1}
The cardinality of each vertex part of $\P_{q,r}(\R)$ is $q^{dr}-q^{d(r-1)}=(1-o(1))q^{dr}$, and $\deg(A)=\deg(B)=q^{(d-1)r}$. The third eigenvalue of $\P_{q,r}(\R)$ is at most $\sqrt{q^{(d-1)(2r-1)}}$.
\end{theorem}  
\subsection{Erd\H{o}s-R\'enyi graphs over finite valuation rings}
For any $\tmmathbf{x} \in \R^d \backslash (\R^{0})^d$, we denote $[\tmmathbf{x}]$  the equivalence class of $\tmmathbf{x}$ in $\R^d \backslash (\R^{0})^d$, where $\tmmathbf{x}, \tmmathbf{y} \in \R^d \backslash (\R^{0})^d$ are equivalent if and only if $\tmmathbf{x} = t \tmmathbf{y}$ for some $t \in \R^{*}$. Let $\E_{q,d}(\R)$ denote the Erd\H{o}s-R\'enny bipartite graph $\E_{q,d}(\R)=(A\cup B,E)$ whose vertices in each part are the points of the projective space over $\R$, where two vertices $[\tmmathbf{x}]$ and $[\tmmathbf{y}]$ are connected if and only if $\tmmathbf{x} \cdot \tmmathbf{y}= 0$. We have the following theorem on the spectrum of $\E_{q,d}(\R)$.
\begin{theorem}[\textbf{Nica}, \cite{bogan}\label{bogan2}] The cardinality of each vertex part of $\E_{q,d}(\R)$ is $ q^{(d-1)(r-1)}(q^{d}-1)/(q-1)$, and $\deg(A)=\deg(B)=q^{(d-2)(r-1)}(q^{d-1}-1)/(q-1)$. The third eigenvalue of $\E_{q,d}(\R)$ is at most $\sqrt{q^{(d-2)(2r-1)}}$.
\end{theorem}
As an application of the Erd\H{o}s-R\'{e}nyi graph $\E_{q,2d}(\R)$, we obtain the following theorem which is a generalization of \cite[Theorem 9]{area}. Some of its applications over finite fields can be found in \cite{hartetal, area}.
\begin{theorem}\label{iosquare}
Let $F$ and $G$ be subsets in $\R^d$. Suppose that $F\cap (\R^0)^d=\emptyset$. Let, for $t\in \R$, 
\[\nu(t):=\left\vert \left\lbrace\right (\mathbf{x},\mathbf{y})\in F\times G\colon \mathbf{x}\cdot \mathbf{y}=t\rbrace\right\vert,\]
where $\mathbf{x}\cdot \mathbf{y}=x_1y_1+\cdots+x_dy_d$.
Then 
\[\sum_{t\in \R}\nu(t)^2\le \frac{|F|^2|G|^2}{q^r}+q^{(d-1)(2r-1)}|F||G|\cdot\max_{\mathbf{x}\in \R^d\setminus (\R^0)^d}|F\cap l_\mathbf{x}|,\]
where
\[l_\mathbf{x}:=\left\lbrace s\mathbf{x}\colon s\in \R^*\right\rbrace,\]
with $\mathbf{x}\in \R^d\setminus (\R^0)^d.$
\end{theorem}
\begin{proof}For any pair of points $(\mathbf{a}, \mathbf{b})\in \R^d\times \R^d$, we define
\[p_{\mathbf{a}, \mathbf{b}}:=(a_1,\ldots,a_d, b_1,\ldots, b_d),\]
and 
\[U:=\left\lbrace p_{\mathbf{x}, -\mathbf{t}}\colon (\mathbf{x}, \mathbf{t}) \in F\times G\right\rbrace\subseteq \R^{2d}, V:=\left\lbrace p_{\mathbf{y}, \mathbf{z}}\colon (\mathbf{y}, \mathbf{z})\in G\times F\right\rbrace
\subseteq \R^{2d}.\]
Since $F\cap (\R^0)^d= \emptyset$,  $U$ and $V$ are sets of points in $\R^{2d}\setminus (\R^0)^{2d}$. It follows from the definition of $\nu(t)$ that $\sum_{t\in \R}\nu(t)^2$ is the number of quadruples $(\mathbf{x}, \mathbf{y}, \mathbf{z}, \mathbf{t})\in F\times G\times F\times G$ satisfying $\mathbf{x}\cdot\mathbf{y}=\mathbf{z}\cdot \mathbf{t}$.   It is clear that  if $\mathbf{x}\cdot\mathbf{y}=\mathbf{z}\cdot\mathbf{t}$, then there is an edge between two vertices $[p_{\mathbf{x}, -\mathbf{t}}]$ and $[p_{\mathbf{y}, \mathbf{z}}]$ in the Erd\H{o}s-R\'{e}nyi graph $\mathcal{E}_{q, 2d}(\R)$. However, we can not make sure that the number of edges between $[U]:=\{[u]\colon u\in U\}$ and $[V]:=\{[v]\colon v\in V\}$ in 
the Erd\H{o}s-R\'{e}nyi graph $\mathcal{E}_{q, 2d}(\R)$ is  an upper bound for the sum $\sum_{t\in \R}\nu(t)^2$, since there might exist two points in $U$ determining the same congruence class, i.e. the same vertex in the Erd\H{o}s-R\'{e}nyi graph $\mathcal{E}_{q, 2d}(\R)$, for example $u\in U$ and $\lambda u\in U$ with $\lambda\in \R^*\setminus \{1\}$.

Thus we will partition  $U$ and $V$ to subsets such that no two points in each subset determine the same vertex in the Erd\H{os}-R\'{e}nyi graph.

Since $m=\max_{\mathbf{x}\in \R^d\setminus (\R^0)^d}|F\cap l_{\mathbf{x}}|$, we can partition $U$ into $m$ subsets $U_1,\ldots, U_m$  of distinct vertices of the Erd\H{os}-R\'{e}nyi graph $\mathcal{ER}_{q, 2d}(\R)$. Similarly, we also can partition $V$ into $m$ subsets $V_1,\ldots, V_m$ of distinct vertices of the Erd\H{os}-R\'{e}nyi graph $\mathcal{ER}_{q, 2d}(\R)$. Then, it is clear that
\[\sum_{t\in \R}\nu(t)^2\le \sum_{i,j}e(U_i,V_j)=\sum_{j=1}^m e(U_1, V_j)+\cdots+\sum_{j=1}^m e(U_m, V_j).\]

On the other hand, for each $1\le i\le m$, it follows from Lemma \ref{expander} and Theorem \ref{bogan2} that 
\begin{eqnarray*}
\sum_{j=1}^m e(U_i, V_j)&\le& \frac{|U_i||V|}{q^r}+q^{(d-1)(2r-1)}\sqrt{|U_i}(\sqrt{|V_1|}+\cdots+\sqrt{|V_m|})\\&\le& \frac{|U_i||V|}{q^r}+\sqrt{m}q^{(d-1)(2r-1)}\sqrt{|U_i|V|},
\end{eqnarray*}
where the second inequality follows from the Cauchy-Schwarz inequality. Thus
\[\sum_{t\in \R}\nu(t)^2\le \frac{|U||V|}{q^r}+q^{(d-1)(2r-1)}m\sqrt{|U||V|}.\]
From the definitions of $U$ and $V$, we get $|U|=|F||G|$ and $|V|=|F||G|$. Therefore the theorem follows.
\end{proof}
Now we prove the following theorem that will be used many times in this paper.
\begin{theorem}\label{manytimes}
Let $F$ and $G$ be subsets in $\R^2$. Suppose that $m=\max_{\mathbf{x}\in \R^2\setminus (\R^0)^2}|F\cap l_\mathbf{x}|$, then we have
\[\left\vert \left\lbrace \mathbf{x}\cdot \mathbf{y}\colon \mathbf{x}\in F, \mathbf{y}\in G\right\rbrace\right\vert =(1-o(1))q^r,\]
when $mq^{(d-1)(2r-1)+r}=o(|F||G|)$.
\end{theorem}
\begin{proof}
We first have
\[\left\vert \left\lbrace\mathbf{x}\cdot \mathbf{y}\colon \mathbf{x}\in F, \mathbf{y}\in G\right\rbrace\right\vert=\left\vert\{t\colon \nu(t)>0\}\right\vert,\] and 
\[\sum_{t\in \R}\nu(t)=|F||G|.\]
On the other hand, let
\[T=\left\vert\{(\mathbf{x}_1,\mathbf{y}_1,\mathbf{x}_2,\mathbf{y}_2)\in  F\times G\times F\times G\colon \mathbf{x}_1\cdot \mathbf{y}_1=\mathbf{x}_2\cdot\mathbf{y}_2\}\right\vert,\]
which implies that 
\[T=\sum_{t\in \R}\nu(t)^2.\]
It follows from the Cauchy-Schwarz inequality that 
\[\left\vert \{t\colon \nu(t)>0\}\right\vert \sum_{t\in \R}\nu(t)^2\ge \left(\sum_{t\in \R}\nu(t)\right)^2.\]
Therefore, we obtain
\begin{equation}\label{ioseq1}\left\vert \{t\colon \nu(t)>0\}\right\vert\ge \frac{|F|^{2}|G|^2}{T}.\end{equation}
On the other hand, it follows from Theorem \ref{iosquare} that
 \begin{equation}\label{ioseq2}T\le \frac{|F|^{2}|G|^2}{q^r}+mq^{(d-1)(2r-1)}|F||G|.\end{equation}
 Putting (\ref{ioseq1}) and (\ref{ioseq2}) together gives us
 \[\left\vert \left\lbrace\mathbf{x}\cdot \mathbf{y}\colon \mathbf{x} \in F, \mathbf{y}\in G\right\rbrace\right\vert
=(1-o(1))q^r,\] 
when $mq^{(d-1)(2r-1)+r}=o(|F||G|)$. This concludes the proof of the theorem.
\end{proof}

\section{Proof of Theorem \ref{muoi}}
A graph $G
= (V, E)$ is called an $(n, d, \lambda)$-graph if it is $d$-regular, has $n$
vertices, and the second eigenvalue of $G$ is at most $\lambda$. Suppose that a graph $G$ is edge-colored by a set of finite colors. We say that $G$ is an $(n,d,\lambda)$-colored graph if the induced subgraph of $G$ on each color is an $(n,d(1+o(1)), \lambda)$-graph. In \cite{vinh-spectral}, the second listed author proved that any large induced subgraph of an $(n,d,\lambda)$-colored graph contains almost all possible colorings of small complete subgraphs.
\begin{theorem}(\cite[Theorem 2.7]{vinh-spectral})\label{vinh-norm-norm}
For any $t\ge 2$. Let $G=(V,E)$ be an $(n,d,\lambda)$-colored graph, and let $m<n$ such that $m\gg \lambda(n/d)^{t/2}$. Suppose that the color set of $\C$ has cardinality $|\C|=(1-o(1))n/d$, then for every subset $U\subset V$ with cardinality $m$, the induced subgraph $G$ on $U$ contains at least $(1-o(1))|\C|^{\binom{t}{2}}$ possible colorings of $K_t$.
\end{theorem}

Let $\C=\R^*$, we now define a graph $G(\R)$ as follows: the vertex set of $G(\R)$ is $\R^d$ and the edge between $\mathbf{x}$ and $\mathbf{y}$ is colored by the $\beta$-color with $\beta\in \C$, if and only of $\mathbf{x}\cdot\mathbf{y}=\beta$. 

One can follow the proof of \cite[Theorem 2.7]{vinh-spectral} step by step by using Lemma \ref{expander},
Theorem \ref{expander2}, and Theorem \ref{bogan1} to get a version of Theorem \ref{vinh-norm-norm} for $G(\R)$ as follows.
\begin{theorem}\label{thm:main9}

For any $t\ge 2$, and  for every subset $U\subset V(G(\R))$ of cardinality $m\gg q^{\frac{(d-1)(2r-1)+rt}{2}}$, the induced subgraph of $G(\R)$ on $U$ contains at least $(1-o(1))q^{r\binom{t}{2}}$ possible colorings of $K_t$.
\end{theorem}

We are now ready to prove Theorem \ref{muoi}.
\begin{proof}[Proof of Theorem \ref{muoi}]
Since $|\E|\gg q^{\frac{(d-1)(2r-1)+r(k+1)}{2}}$, it follows from Theorem \ref{thm:main9} that the induced subgraph $G(\R)$ on $\E$ contains at least $(1-o(1))q^{r\binom{k+1}{2}}$ possible colorings of $K_{k+1}$. Moreover, each coloring is corresponding to a dot-product congruence class, thus the number of dot-product congruence classes of $k$-simplices in $\E$ is at least $(1-o(1))q^{r\binom{k+1}{2}}$. This ends the proof of the theorem.
\end{proof}

\section{Proofs of Theorems \ref{mainthm297} and \ref{co2}}

Before giving a proof of Theorem \ref{mainthm297}, we have the following observation:  Since the area of triangle is invariant under translations, we can assume that $\mathbf{0}\in \E$, and the formula of area of the triangle formed by three vertices $\mathbf{0}$, $\mathbf{a}=(a_1, a_2)$ and $\mathbf{b}=(b_1, b_2)$ is $a_1b_2-a_2b_1$. Let $S$ be the set of triangles in $\E$ which share a common vertex at $\mathbf{0}$. Then the number of distinct areas of triangles in $S$ is at least the cardinality of 
\[\E\cdot \E':=\{\mathbf{x}\cdot\mathbf{y}\colon \mathbf{x}\in \E, \mathbf{y}\in \E'\},\]
where $\E'=\{(y, -x)\colon (x, y)\in \E\}$. 

A result of Nica in \cite{bogan}, i.e. Theorem \ref{nc}, states that if $|\E||\E'|> q^{4r-1}$, then $|\E\cdot\E'|\gg q^r-q^{r-1}$. It is clear that $|\E|=|\E'|$. Thus if $|\E|>q^{2r-\frac{1}{2}}$, then the number of distinct areas of triangles in $\E$ is at least $q^r-q^{r-1}$. In fact, the result of Nica \cite{bogan} gives us even more information, for instance, the number of triangles of area $t\in \R^*$ is at least $(1-o(1))\frac{|\E|^2}{q^r}$ when $q^{2r-\frac{1}{2}}=o(|\E|)$.

However, in order to decrease from the exponent $q^{2r-\frac{1}{2}}$ to $q^{2r-1}$, we need to use more complicated and tricky arguments. First we need to prove the following lemma.

\begin{lemma}\label{lm1297}
Let $\R$ be a finite valuation ring of order $q^r$, and  let $\E$ be a set of $8q^{2r-1}$ points in $\R^2$. Then there exists a point $\mathbf{z}$ of $\E$ such that $\mathbf{z}$ is contained in at least $q^r/8$ lines, and each of these lines passes through at least $q^{r-1}+1$ points from $\E$.
\end{lemma}
\paragraph{Proof of Lemma \ref{lm1297}:}
To prove Lemma \ref{lm1297}, we make use of the following theorem on the number of incidences between points and lines in $\R^2$, where a line in $\R^2$ is of the form 
\[ax+by+c=0, ~(a,b,c)\in R^3\setminus (\R^0)^3.\]
\begin{theorem}\label{mot}
Let $\R$ be a finite valuation ring of order $q^r$,  $\E$ be a set of points in $\R^2$ and $\mathcal{L}$ be a set of lines in
  $\R^2$. Then the number of incidences between the point set $\E$ and the line set $\L$, denoted by $I(\E, \L)$, satisfies
  \begin{equation}
  \left\vert I(\E,\L)-\frac{|\E||\mathcal{L}|}{q^r}\right\vert \le 
q^{(2r-1)/2}\sqrt{|\E| |\mathcal{L}|}.\nonumber
  \end{equation}
\end{theorem}
\begin{proof}
We identify each point $(x_1,x_2)\in \E$ with a vertex $[x_1,x_2,1]$ of the Erd\H{o}s-R\'{e}nyi graph $\mathcal{E}_{q,3}(\R)$. Let $\E'$ be the set of corresponding vertices. Similarly, we identify each line $ax+by=c$ in $\L$, $(a,b,c)\not\in (\R^0)^3$, with a vertex $[a, b, -c]$ of the Erd\H{o}s-R\'{e}nyi graph $\mathcal{E}_{q,3}(\R)$. Let $\L'$ be the set of corresponding vertices. Then $\E'$ and $\L'$ are sets of distinct vertices with $|\E'|=|\E|$ and $|\L'|=|\L|$.

It is easy to see that the number of incidences between $\E$ and $\L$ equals the number of edges between $\E'$ and $\L'$ in the Erd\H{o}s-R\'{e}nyi graph $\mathcal{E}_{q,3}(\R)$. It follows from Lemma \ref{expander} and Theorem \ref{bogan2} that
\[\left\vert I(\E,\L)-\frac{|\E||\L|}{q^r}\right\vert\le q^{(2r-1)/2}\sqrt{|\E||\L|}.\]
This concludes the proof of the theorem.
\end{proof}
The following is a corollary of Theorem \ref{mot}.
\begin{corollary}\label{hai297}
Let $\R$ be a finite valuation ring of order $q^r$, and let $\E$ be a set of $3q^{2r-1}$ points in $\R^2$. Then the number of distinct lines spanned by $\E$ containing at least $q^{r-1}+1$ points from $\E$ is at least $q^{2r}/4$.
\end{corollary}
\begin{proof}
Let $\L_1$ be the set of lines in $\R^2$ such that each line contains at most $q^{r-1}$ points from $\E$. We now show that $|\L_1|\le 3q^{2r}/4$. Indeed, we first have $I(\E, \L_1)\le  q^{r-1}|\L_1|$, and it follows from Theorem \ref{mot} that
\[I(\E, \L_1)\ge \frac{|\E||\L_1|}{q^r}-q^{(2r-1)/2}\sqrt{|\E||\L_1|}\ge 3q^{r-1}|\L_1|-\sqrt{3}q^{2r-1}\sqrt{|\L_1|}.\]
This implies that 
\[2q^{r-1}|\L_1|\le \sqrt{3}q^{2r-1}\sqrt{|\L_1|}.\]
Thus, we obtain 
\[|\L_1|\le \frac{3q^{2r}}{4}.\]
On the other hand, the number of lines of the form $y=ax+b$ in $\R^2$ is $q^{2r}$, then the number of lines of the form $y=ax+b$ containing at least $q^{r-1}+1$ points from $\E$ is at least $q^{2r}/4$. Since any two lines in $\R^2$ have at most $q^{r-1}$ points in common, these lines are distinct. This completes the proof of the corollary.
\end{proof}
We are ready to give a proof of Lemma \ref{lm1297}.
\begin{proof}[Proof of Lemma \ref{lm1297}]
Let $\L$ be the set of lines in $\R^2$ such that each line in $\L$ contains at least $q^{r-1}+1$ points from $\E$. It follows from Corollary \ref{hai297} that $|\L|\ge q^{2r}/4$. From the lower bound of Theorem \ref{mot}, we have if $|\E||\L|\ge 2q^{4r-1}$, then $I(\E, \L)\ge |\E||\L|/q^r$. Thus it implies that $I(\E, \L)\ge q^{3r-1} $. Therefore, by the pigeon-hole principle, there exists a point $\mathbf{z}\in \E$ such that $\mathbf{z}$ is contained in at least $q^{r}/8$ lines from $\L$, and each of these lines contains at least $q^{r-1}+1$ points from $\E$.
\end{proof}
\paragraph{Proofs of Theorems \ref{mainthm297} and \ref{co2}:}
\begin{proof}[Proof of Theorem \ref{mainthm297}]Since $|\E|\ge 8q^{2r-1}$, we have$|(\R^0)^2|=o\left(|\E\cap \R^2\setminus \left(\R^0\right)^2|\right)$. Thus, without loss of generality, we can assume that $\E\subseteq \R^2\setminus \left(\R^0\right)^2$. Lemma \ref{lm1297} implies that there exists a point $\mathbf{z}\in \E$ such that $\mathbf{z}$ is contained in at least $q^r/8$ lines, and each of these lines passes through least $q^{r-1}+1$ points from $\E$. We denote the set of these lines by $\L'$.
 
We now consider the set of triangles in $\E$ which share a common vertex at $\mathbf{z}$.  Since the area of a triangle is invariant under translations, we assume that $\mathbf{z}=\mathbf{0}$, and all lines in $\L'$ are of the form $l_k:=\{y=kx\}$ with $k\in \R$. It is easy to see that for a fixed $a\in \R$, the points $(x, ax)\not\in l_b$ for all $b\ne a$ and $x\in \R^*$. Thus, we can choose
$q^r/8$ points of $\E$ from the lines in $\L'$ such that no two points belong to the same line. Let $F$ be the set of such points, and $G:=\left\lbrace (-p_2, p_1)\colon (p_1, p_2)\in \E\right\rbrace$. Then the number of distinct areas of triangles formed by three vertices $(\mathbf{0}, \mathbf{a}, \mathbf{b})\in \{\mathbf{0}\}\times F\times G$ is the cardinality of the set $F\cdot G=\left\lbrace \mathbf{a}\cdot \mathbf{b}\colon \mathbf{a}\in F, \mathbf{b}\in G\right\rbrace$.

Applying Theorem \ref{manytimes} with $|F|=q^{r}/8, |G|=8q^{2r-1}$, $d=2$, and $m=1$, we get
\[\left\vert F\cdot G\right\vert =(1-o(1))q^r.\]
This implies that the number of distinct areas determined by $\E$ is at least $(1-o(1))q^r$. This concludes the proof of the theorem.
\end{proof}
\begin{proof}[Proof of Theorem \ref{co2}]
We prove Theorem \ref{co2} by induction on $d$. The base case $d=2$ follows from Theorem \ref{mainthm297}. Suppose that the statement is true for all $2<i\le d-1$, we now show that it also holds for $d$. Indeed, since $|\E|\ge 8q^{r-1}q^{r(d-1)}$, there exists a hyperplane $H_t:=\{\mathbf{x}\in \R^d\colon x_d=t\}$ such that $|\E\cap H_t|\ge 8q^{r-1}q^{r(d-2)}$. By induction hypothesis, we have $|V_{d-1}(\E\cap H_t)| \ge q^r/2$. 

Since $V_d(\E)$ is invariant under translations, we can assume that $t=0$. Moreover, the number of points of $\E$ satisfying $x_d\in \R^0$ is at most $q^{r-1}\cdot q^{r(d-1)}$. Thus, there exists a point $\mathbf{z}\in \E$ such that $z_d\in \R^*$. On the other hand, $V_d^\mathbf{z}(\E)$ are determinants of size $d+1$ of the form
\[\det\begin{pmatrix}
1&\ldots&1&1\\
x_1^1&\ldots &x_1^d&z_1\\
 \vdots  & \ddots  & \vdots & \vdots\\
x_{d-1}^1&\ldots &x_{d-1}^d&z_{d-1}\\
0&\ldots&0&z_d\\
\end{pmatrix}= z_d\cdot \det\begin{pmatrix}
1&\ldots&1\\
x_1^1&\ldots &x_1^d\\
\vdots&\cdots&\vdots\\
x_{d-1}^1&\ldots &x_{d-1}^d
\end{pmatrix}.\]
This completes the proof of the Theorem \ref{co2}.
\end{proof}
\section{Proof of Theorem \ref{perper}}
Since $|\A|\gg q^{\frac{(k-1)(2r-1)+r}{2k-1}}>q^{r-1}=|\R^0|$, there exists a unit $u\in \A\cap \R^*$. Let $\mathbf{1}:=(1, \ldots, 1)\in \R^k$, $\mathbf{u}=(u, \ldots, u)\in \R^k$. For any two points $\mathbf{x}$ and $\mathbf{y}$ in $\A^k$, let $M(\mathbf{u}, \mathbf{x}, \mathbf{y})$ denote the matrix whose rows are $\mathbf{x}, \mathbf{y}$ and $(k-2)$ $\mathbf{u}$'s.

We have 
\begin{equation}\label{317}\mathtt{Per}(M(\mathbf{u}, \mathbf{x}, \mathbf{y}))=u^k\mathtt{Per}\left(M(\mathbf{1}, \mathbf{x}/u, \mathbf{y}/u)\right)=u^k\sum_{i=1}^k\frac{x_i}{u}\sum_{j\ne i}\frac{y_j}{u}.\end{equation}

Thus we are able to reduce the permanent problem to the dot product problem of two following sets:
\[F:=\left\lbrace \left(\frac{x_1}{u},\ldots, \frac{x_k}{u}\right)\colon (x_1, \ldots, x_k)\in \A^k\right\rbrace,\]

\[G:=\left\lbrace \left(\sum_{j\ne 1}\frac{y_j}{u},\ldots, \sum_{j\ne k}\frac{y_j}{u}\right)\colon (y_1, \ldots, y_k)\in \A^k\right\rbrace.\]

It is clear that $|F|=|\A|^k$ and $|G|=|\A|^k$ since $\gcd(k, q^r)=1$. From (\ref{317}) we get  
\[|\mathtt{Per}(M(\mathbf{u}, \mathbf{x}, \mathbf{y}))|=|F\cdot G|,\] 
then the theorem follows immediately from Theorem \ref{manytimes} with $m=|\A|$.
\section{Proof of Theorem \ref{r1}}
The proof of Theorem \ref{r1} is based on the study of the equation $(x_1/2 - z)(x_1/2 + z) = x_2$ where $x_1 \in X_1, x_2 \in X_2$ and $z \in X_3$. Here, $X_3 \equiv \R^{*}$. This equation is equivalent to the equation $(x_1/2)^2 - x_2 = z^2$. We set
\begin{eqnarray*}
A_1 & = & \{ (x_1/2)^2 \mid x_1 \in X_1\}, A_2 =  \{ -x_2 \mid x_2 \in X_2\},\\
A_3 & = & \{ z^2 \mid z \in X_3\}, A_4 = \{ z^2 \mid z \in X_3\}.
\end{eqnarray*}
Note that the equation $x^2=a^2$ has at most two solutions in $\R$ for any $a\in \R^*$. Thus, we have 
\[|A_1|\ge |X_1|/2, ~|A_2|=|X_2|, |A_3|\ge |X_3|/2, |A_4|\ge |X_3|/2.\]
The equation $(x_1/2)^2-x_2=z^2$ has a solution $x_1\in X_1, x_2\in X_2, z\in X_3$ if and only if there exists an edge between two vertex sets $U:=\{[a_3, 1, a_1]\colon (a_3, 1, a_1)\in A_3\times \{1\}\times A_1\}$ and $V:=\{[a_4, 1, a_2]\colon (a_4, 1, a_2)\in A_4\times \{1\}\times A_2\}$ in the Erd\H{o}s-R\'{e}nyi graph $\mathcal{E}_{q, 3}(\R)$. Therefore, from Lemma \ref{expander} and Lemma \ref{bogan2} that 
\[e(U, V)\ge \frac{|A_1||A_2||A_3||A_4|}{q^r}-q^{(2r-1)/2}\sqrt{|A_1||A_2||A_3||A_4|}.\]
Thus if  \[|X_1||X_2|>\frac{q^{4r-1}}{(q^r-q^{r-1})^2},\] then $e(U, V)>0$, and the theorem follows.

\end{document}